\documentclass{birkjour}

\def\gr{\operatorname{gr}}

\def\id{\operatorname{id}}

\def\red{\operatorname{red}}

\def\pr{\operatorname{pr}}

\def\mod{\operatorname{mod}}

 \newtheorem{thm}{Theorem}[section]
 
 \newtheorem{lem}[thm]{Lemma}
 
 \theoremstyle{definition}
 
 \theoremstyle{remark}

 \numberwithin{equation}{section}

\begin{document}

\title[A classification theorem]
 {A classification theorem and a spectral\\
  sequence for a locally
  free sheaf\\
  cohomology of a
supermanifold}

\author[E.G.~Vishnyakova]{E.G.~Vishnyakova}

\address{%
Max-Planck-Institut f\"{u}r Mathematik\\
P.O.Box: 7280\\
53072 Bonn\\
Germany
}

\email{vishnyakovae@googlemail.com, Liza@mpim-bonn.mpg.de}

\thanks{This work was partially supported by MPI Bonn, SFB TR $|$$12$, DFG 1388 and by the Russian Foundation
for Basic Research (grant no. 11-01-00465a).}



\date{January 1, 2004}

\dedicatory{To our team coach Yu.A. Kirillov on his 70th birthday}

\begin{abstract}
This paper is based on the paper \cite{Oni_Vi},
where two classification theorems for locally free sheaves on supermanifolds were proved and a spectral
sequence for a locally free sheaf
of modules $\mathcal E$ was obtained.
We consider another filtration of the locally free sheaf $\mathcal E$, the corresponding classification theorem and  the
spectral sequence, which is more convenient in some cases. The methods, which we are using here, are similar to \cite{Oni_SPS, Oni_Vi}.

 The first spectral sequence of this kind was constructed by A.L.~Oni\-shchik in \cite{Oni_SPS} for the tangent sheaf of a supermanifold.
 However, the spectral sequence considered in this paper is not a generalization of Onishchik's spectral sequence from \cite{Oni_SPS}.
\end{abstract}

%%% ----------------------------------------------------------------------
\maketitle
%%% ----------------------------------------------------------------------

\section{ Main definitions and classification theorems}

\subsection{ Main definitions} Let $(M,\mathcal{O})$ be a supermanifold of
dimension $n|m$, i.e. a $\mathbb{Z}_2$-graded ringed space that is locally
isomorphic to a superdomain in $\mathbb{C}^{n|m}$. The underlying complex
manifold $(M,\mathcal{F})$ is called the {\it reduction} of $(M,\mathcal{O})$.
The simplest class of supermanifolds constitute the so-called {\it split supermanifolds}.
We recall that a supermanifold $(M,\mathcal{O})$ is called split if
$\mathcal{O}\simeq \bigwedge_{\mathcal{F}}\mathcal{G}$, where $\mathcal{G}$ is a locally
free sheaf of $\mathcal{F}$-modules on $M$. With any supermanifold $(M,\mathcal{O})$ one
can associate a split supermanifold $(M,\tilde{\mathcal{O}})$ of the same dimension which is
called the {\it retract} of $(M,\mathcal{O})$. To construct it, let us consider the
$\mathbb{Z}_2$-graded sheaf of ideals $\mathcal{J} =
\mathcal{J}_{\bar 0}\oplus\mathcal{J}_{\bar 1}\subset\mathcal{O}$ generated by odd elements
of $\mathcal{O}$. The structure sheaf of the retract is defined by
$$
\tilde{\mathcal{O}}= \bigoplus_{p\geq 0} \tilde{\mathcal{O}}_p,\,\,\,\text{where} \,\,\,
\tilde{\mathcal{O}}_p= \mathcal{J}^p/ \mathcal{J}^{p+1},\,\,\,\mathcal{J}^0:= \mathcal{O}.
$$
Here $\tilde{\mathcal{O}}_1$ is a locally free sheaf of $\mathcal{F}$-modules on $M$ and
$\tilde{\mathcal{O}}_p = \bigwedge^p_{\mathcal{F}}\tilde{\mathcal{O}}_1$.
By definition, the following sequences
\begin{equation}\label{exact sequence structure sheaf}
\begin{array}{c}
0\to \mathcal{J}\cap \mathcal{O}_{\bar 0} \to \mathcal{O}_{\bar 0}\stackrel{\pi}{\to} \tilde{\mathcal{O}}_{0} \to 0,\\
0\to \mathcal{J}^2\cap \mathcal{O}_{\bar 1} \to \mathcal{O}_{\bar 1} \stackrel{\tau}{\to} \tilde{\mathcal{O}}_{1} \to 0.
\end{array}
\end{equation}
are exact. Moreover, they are locally split. The supermanifold $(M,\mathcal{O})$ is split iff both sequences are globally split.

Denote by $\mathcal{S}_{\bar 0}$ and $\mathcal{S}_{\bar 1}$ the even and the odd parts of a $\mathbb{Z}_2$-graded sheaf of
$\mathcal{O}$-modules $\mathcal{S}$ on $M$, respectively; by
$\Pi(\mathcal{S})$ we denote the same sheaf of $\mathcal{O}$-modules
$\mathcal{S}$ equipped with the following $\mathbb{Z}_2$-grading:
$\Pi(\mathcal{S})_{\bar 0} = \mathcal{S}_{\bar 1}$, $
\Pi(\mathcal{S})_{\bar 1} = \mathcal{S}_{\bar 0}$.
A $\mathbb{Z}_2$-graded sheaf of $\mathcal{O}$-modules on $M$ is
called {\it free\/} ({\it locally free}) {\it of rank\/} $p|q,\, p,q\ge 0$ if it
is isomorphic (respectively, locally isomorphic) to the $\mathbb{Z}_2$-graded
sheaf of $\mathcal{O}$-modules $\mathcal{O}^p\oplus\Pi(\mathcal{O})^q$. For
example, the tangent sheaf $\mathcal{T}$ of a supermanifold $(M,\mathcal{O})$
is a locally free sheaf of $\mathcal{O}$-modules.

Let now $\mathcal{E}=\mathcal{E}_{\bar 0}\oplus \mathcal{E}_{\bar 1}$ be a
locally free sheaf of $\mathcal{O}$-modules of rang $p|q$ on an arbitrary
supermanifold $(M,\mathcal{O})$. We are going to construct a locally free sheaf
of the same rank on $(M,\tilde{\mathcal{O}})$.
First, we note that $\mathcal{E}_{\red}:=\mathcal{E}/\mathcal{J}\mathcal{E}$ is
a locally free sheaf of $\mathcal{F}$-modules on $M$. Moreover,
$\mathcal{E}_{\red}$ admits the $\mathbb{Z}_2$-grading
$
\mathcal{E}_{\red} = (\mathcal{E}_{\red})_{\bar 0}\oplus(\mathcal{E}_{\red})_{\bar 1},
$
by two locally free sheaves of $\mathcal{F}$-modules
$$
(\mathcal{E}_{\red})_{\bar 0}:= \mathcal{E}_{\bar
0}/\mathcal{J}\mathcal{E} \cap \mathcal{E}_{\bar 0}\,\, \text{ and}
\,\, (\mathcal{E}_{\red})_{\bar 1}:= \mathcal{E}_{\bar
1}/\mathcal{J}\mathcal{E} \cap \mathcal{E}_{\bar 1}
$$
of ranks $p$ and $q$, respectively.
Further, the sheaf $\mathcal{E}$ possesses the filtration
\begin{equation}\label{filtr_rassloenie}
\mathcal{E}=\mathcal{E}_{(0)}\supset \mathcal{E}_{(1)} \supset
\mathcal{E}_{(2)} \supset \ldots , \,\,\text{where}\,\,\mathcal{E}_{(p)} = \mathcal{J}^{p} \mathcal{E}_{\bar 0} + \mathcal{J}^{p-1} \mathcal{E}_{\bar 1},\,\, p\geq 1.
\end{equation}
Using this filtration, we can construct the following locally free sheaf of
$\tilde{\mathcal{O}}$-modules on $M$:
$$
\begin{array}{c}
\tilde{\mathcal{E}}= \bigoplus_p \tilde{\mathcal{E}}_p,
\,\,\,\text{where} \,\,\,
 \tilde{\mathcal{E}}_p= \mathcal{E}_{(p)} /
\mathcal{E}_{(p+1)}.
\end{array}
$$
The sheaf $\tilde{\mathcal{E}}$ is also a locally free sheaf
of $\mathcal F$-modules. In other words, $\tilde{\mathcal{E}}$ is a sheaf of sections of a certain vector bundle.
The following exact sequence gives a description of $\tilde{\mathcal{E}}$.
$$
0\to \tilde{\mathcal{O}}_p\otimes (\mathcal E_{\red})_{\bar 0} \to \tilde{\mathcal{E}}_p \to
\tilde{\mathcal{O}}_{p-1}\otimes (\mathcal E_{\red})_{\bar 1} \to 0.
$$
We also have the following two exact sequences, which
are locally split:
\begin{equation}\label{sequence_stuct bundle}
\begin{array}{c}
0\to\mathcal{E}_{(1)\bar 0} \to
\mathcal{E}_{(0)\bar 0}\stackrel{\alpha}{\to} \tilde{\mathcal{E}}_0
 \to 0;\\
0\to \mathcal{E}_{(2)\bar 1} \to
\mathcal{E}_{(1)\bar 1} \stackrel{\beta}{\to} \tilde{\mathcal{E}}_1 \to 0.
\end{array}
\end{equation}
The sheaf $\tilde{\mathcal{E}}$ is $\mathbb{Z}$-graded by definition. Unlike the $\mathbb{Z}_{2}$-grading considered in \cite{Oni_Vi}, the natural
$\mathbb{Z}_{2}$-grading is compatible with this $\mathbb{Z}$-grading.
$$
(\tilde{\mathcal{E}})_{\bar 0}: = \bigoplus_{p=2k}  \tilde{\mathcal{E}}_p,\quad
 (\tilde{\mathcal{E}})_{\bar 1}: = \bigoplus_{p=2k_1}  \tilde{\mathcal{E}}_p.
$$

\subsection{Classification theorem for locally free sheaves $\mathcal{E}$
on supermanifolds with given $\tilde{\mathcal{E}}$}

Our objective now is to classify locally free sheaves $\mathcal{E}$ of
$\mathcal{O}$-modules on supermanifolds $(M,\mathcal{O})$ which
have the fixed retract $(M,\tilde{\mathcal{O}})$ and such that the corresponding locally free sheaf $\tilde{\mathcal{E}}$ is fixed.

Let $(M,\mathcal{O})$ and $(M,\mathcal{O}')$ be two supermanifolds,
$\mathcal{E}$, $\mathcal{E}'$ be locally free sheaves of
$\mathcal{O}$-modules and $\mathcal{O}'$-modules on $M$, respectively.
Suppose that $\Psi: \mathcal{O} \to \mathcal{O}'$ is a superalgebra sheaf morphism.
 A vector space sheaf morphism $\Phi_{\Psi}: \mathcal{E}\to
\mathcal{E}'$ is called a {\it quasi-morphism} if
$$
\Phi_{\Psi}(fv)= \Psi(f)\Phi_{\Psi}(v), \,\,f\in \mathcal{O}, \,\,v\in \mathcal{E}.
$$
As usual, we assume that $\Phi_{\Psi}(\mathcal{E}_{\bar i})
\subset \mathcal{E}'_{\bar i}$, $\bar i\in \{\bar 0, \bar 1\}$.
An invertible quasi-morphism is called a {\it quasi-isomorphism}. A quasi-isomorphism $\Phi_{\Psi}: \mathcal{E}\to
\mathcal{E}$ is also called a {\it quasi-automorphism} of
$\mathcal{E}$.
Denote by $\mathcal{A}ut \mathcal{E}$ the sheaf of quasi-automorphisms of $\mathcal{E}$.
It has a double filtration by the subsheaves
$$
\begin{array}{c}
\mathcal{A}ut_{(p)(q)} \mathcal{E}:= \{\Phi_{\Psi}\in \mathcal{A}ut
\mathcal{E}\mid \Phi_{\Psi}(v)\equiv v \mod \mathcal{E}_{(p)},\,
\Psi(f)=f
\mod \mathcal{J}^q\\
\text{for}\,\, v\in \mathcal{E}, f\in \mathcal{O} \},\,\,p,q\geq 0.
\end{array}
$$
We also define the following subsheaf of $\mathcal{A}ut \tilde{\mathcal{E}}$:
\begin{equation}
\begin{split}
\widetilde{\mathcal{A}ut} \tilde{\mathcal{E}}:= \{ \Phi_{\Psi} \mid
\Phi_{\Psi}\in \mathcal{A}ut (\tilde{\mathcal{E}}), \,\,
\Phi_{\Psi}\text{ preserves the $ \mathbb{Z}$-grading of $\tilde{\mathcal{E}}$}\}.
\end{split}
\end{equation}
If $\Phi_{\Psi}\in \widetilde{\mathcal{A}ut} \tilde{\mathcal{E}}$, then $\Psi: \tilde{\mathcal{O}}\to \tilde{\mathcal{O}}$ also preserves
the $\mathbb{Z}$-grading.
The $0$-th cohomology group $H^0(M,\widetilde{\mathcal{A}ut} \tilde{\mathcal{E}})$
acts on the sheaf $\mathcal{A}ut \tilde{\mathcal{E}}$ by the
automorphisms $\delta\mapsto a\circ \delta \circ a^{-1}$, where
$a\in H^0(M,\widetilde{\mathcal{A}ut} \tilde{\mathcal{E}})$ and
$\delta\in \mathcal{A}ut \tilde{\mathcal{E}}$. It is easy to see that
this action leaves invariant the subsheaves $\mathcal{A}ut_{(p)(q)}
\tilde{\mathcal{E}}$ and hence induces an action of
$H^0(M,\widetilde{\mathcal{A}ut} \tilde{\mathcal{E}})$ on the
cohomology set $H^1(M,\mathcal{A}ut_{(p)(q)} \tilde{\mathcal{E}})$. The unit element
$\epsilon\in H^1(M,\mathcal{A}ut_{(p)(q)}
\mathcal{E}')$ is a fixed point with respect to the action of
$H^0(M,\mathcal{A}ut \mathcal{E}')$.

Let $\mathcal{E}$ be a locally free sheaf of $\mathcal{O}$-modules on $M$.
Denote
$$
[\mathcal{E}] = \{ \mathcal{E}' \mid \mathcal{E}'\,\, \text{is quasi-isomorphic to}\,\, \mathcal{E}  \}.
$$
The total space of the bundle corresponding to a locally free sheaf $\mathcal{E}$ will be denoted $\mathbb{E}$. It is a supermanifold. The locally free sheaf $\tilde{\mathcal{E}}$ corresponding to $\mathcal{E}$ has the following property: The retract $\tilde{\mathbb{E}}$ of $\mathbb{E}$ is the total space of the bundle corresponding to $\tilde{\mathcal{E}}$.

\begin{thm}\label{teor_main} Let $(M,\mathcal{O}')$ be a split supermanifold and
$\mathcal{E}'$ be a locally free sheaf of
$\mathcal{O}'$-modules on $M$ such that $\mathcal{E}' \simeq \tilde{\mathcal{E}}'$.
Then
$$
\begin{array}{c}
\{[\mathcal{E}] \mid \tilde{\mathcal{O}}= \mathcal{O}', \,\, \tilde{\mathcal{E}} =
\mathcal{E}'\}\stackrel{1:1}{\longleftrightarrow}
 H^1(M,\mathcal{A}ut_{(2)(2)} \mathcal{E}')/H^0(M,\widetilde{\mathcal{A}ut} \mathcal{E}').
\end{array}
$$
The orbit of the unit element $\epsilon$, which is $\epsilon$ itself, corresponds to $\mathcal{E}'.$

\end{thm}

\begin{proof}  Let $\mathcal{E}$ be a locally free sheaf of
$\mathcal{O}$-modules on $(M,\mathcal{O})$ and $\mathcal{U}=\{ U_i\}$ be an
open covering of $M$ such that (\ref{exact sequence structure sheaf}) and
(\ref{sequence_stuct bundle}) are split (hence exact) over $U_i$ and $\mathcal{E}|_{U_i}$ are
free. In this case, $\tilde{\mathcal{E}}|_{U_i}$ are free sheaves of
$\tilde{\mathcal{O}}$-modules. We fix homogeneous bases (even and odd, respectively) $(\hat{e}^i_j)$ and $(\hat{f}^i_j)$ of the free sheaves of $\tilde{\mathcal{O}}$-modules
$\tilde{\mathcal{E}}|_{U_i}$, $U_i\in \mathcal{U}$. Without loss of generality, we may assume that $\hat{e}^i_j\in \tilde{\mathcal{E}}_0$ and $\hat{f}^i_j\in \tilde{\mathcal{E}}_1$.
 We are going to define an isomorphism
$\delta_i:\mathcal{E}|_{U_i}\to \tilde{\mathcal{E}}|_{U_i}$.

 Let
$e_j^i\in \mathcal{E}_{(0)\bar 0}$ be such that
$\alpha(e^i_j)=\hat{e}^i_j$ and $f_j^i\in \mathcal{E}_{(0)\bar 1}$
be such that $\beta(f^i_j)=\hat{f}^i_j$, see (\ref{sequence_stuct bundle}). Then $(e_j^i,f_j^i)$ is a
local basis of $\mathcal{E}|_{U_i}$. A splitting of
(\ref{exact sequence structure sheaf}) determines a local isomorphism
$\sigma_i:\mathcal{O}|_{U_i} \to \tilde{\mathcal{O}}|_{U_i}$, see \cite{Green}. We put
$$
\delta_i(\sum h_je^i_j+ \sum g_jf^i_j)=  \sum
\sigma_i(h_j)\hat{e}^i_j+ \sum\sigma_i(g_j)\hat{f}^i_j,
\,\,\,h_j,g_j\in \mathcal{O}.
$$
Obviously, $\delta_i$ is an isomorphism. We put $\gamma_{ij}:=
\sigma_i\circ \sigma_j^{-1}$ and $(g_{ij})_{\gamma_{ij}}:=
\delta_i\circ \delta_j^{-1}$. Moreover, $(\gamma_{ij})\in
Z^1(\mathcal{U}, \mathcal{A}ut_{(2)} \tilde{\mathcal{O}})$, see \cite{Green} for more details.  We want to show that
$$
((g_{ij})_{\gamma_{ij}})\in Z^1(\mathcal{U}, \mathcal{A}ut_{(2)(2)}
\tilde{\mathcal{E}}).
$$
Let us take $v\in \tilde{\mathcal{E}}|_{U_j}$, $v = \sum h_k\hat{e}^i_k+ \sum g_k\hat{f}^i_k,
\,\,\,h_j,g_j\in \tilde{\mathcal{O}}$. Then by definition we have
$$
\delta_j^{-1}(v) = \sum
\sigma_j^{-1}(h_k)e^j_k+ \sum\sigma_j^{-1}(g_k)f^j_k.
$$
The transition functions of $\tilde{\mathcal{E}}$ may be expressed in $U_i\cap U_j$ as follows:
$$
e^j_k = \sum a^k_s e^i_s+ \sum b^k_s f^i_s,\,\,
f^j_k = \sum c^k_s e^i_s+ \sum d^k_s f^i_s, \,\,a^k_s,d^k_s\in \mathcal{O}_{\bar 0}, \,\,
b^k_s,c^k_s\in \mathcal{O}_{\bar 1}.
$$
Further,
$$
\alpha (e^j_k) = \hat{e}^j_k = \sum \pi(a^k_s) \hat{e}^i_s, \quad
\beta(f^j_k) = \hat{f}^j_k =
\sum \tau(c^k_s) \hat{e}^i_s+ \sum \pi(d^k_s) \hat{f}^i_s.
$$
We have
$$
\begin{array}{rl}
\delta_j\circ \delta_j^{-1}(v) =&\sum _k \gamma_{ij}(h_k) (\sum_s \sigma_i(a^k_s) \hat{e}^i_s+ \sum_r \sigma_i(b^k_s) \hat{f}^i_s)+ \\
& \sum _k \gamma_{ij}(g_k) (\sum_s \sigma_i(c^k_s) \hat{e}^i_s+ \sum_s \sigma_i(d^k_s) \hat{f}^i_s)=\\
&\sum _k h_k (\sum_s \pi(a^k_s) \hat{e}^i_s)+  \sum _k g_k (\sum_s \tau(c^k_s) \hat{e}^i_s+\\
& \sum_s \pi(d^k_s) \hat{f}^i_s)\mod\tilde{\mathcal{E}}_{(2)} = v \mod\tilde{\mathcal{E}}_{(2)}.
\end{array}
$$
The rest of the proof is a direct repetition of the proof of Theorem~2 from \cite{Oni_Vi}.
\end{proof}

\section{The spectral sequence}

\subsection{Quasi-derivations}

Quasi-derivations were defined in \cite{Oni_Vi}. Let us briefly recall that construction.
Consider a locally free sheaf $\mathcal{E}$ on a supermanifold $(M,\mathcal{O})$. An even vector
space sheaf morphism $A_{\Gamma}:\mathcal{E} \to \mathcal{E}$ is
called a {\it quasi-derivation} if $A_{\Gamma}(fv)=\Gamma(f)v+f
A_{\Gamma}(v)$, where $f\in \mathcal{O}$, $v\in \mathcal{E}$ and $\Gamma$ is a certain even super vector field. Denote by $Der\mathcal{E}$ the sheaf of quasi-derivations. It is a sheaf of Lie algebras with respect to the commutator
$[A_{\Gamma}, B_{\Upsilon}]:= A_{\Gamma}\circ B_{\Upsilon} -
B_{\Upsilon} \circ A_{\Gamma}$. The sheaf $Der\mathcal{E}$
possesses a double filtration
$$
\begin{array}{cccc}
Der_{(0)(0)}\mathcal{E} &\supset & Der_{(2)(0)}\mathcal{E}& \supset\cdots \\
\cup &  & \cup& \\
Der_{(0)(2)}\mathcal{E} &\supset & Der_{(2)(2)}\mathcal{E}& \supset\cdots\\
\vdots &  & \vdots& \\
\end{array},
$$
where
$$
\begin{array}{c}
Der_{(p)(q)}\mathcal{E}:= \{ A_{\Gamma}\in Der\mathcal{E} \mid
A_{\Gamma}(\mathcal{E}_{(r)})\subset \mathcal{E}_{(r+p)},\,\,
\Gamma(\mathcal{J}^s)\subset \mathcal{J}^{s+q},\,\, r,s\in \mathbb{Z}\},
\end{array}
$$
where $p,q\geq 0$.
The map defined by the usual exponential series
$$
\exp: Der_{(p)(q)}\mathcal{E} \to \mathcal{A}ut_{(p)(q)} \mathcal{E}, \,\, p,q\geq 2,
$$
is an isomorphism of sheaves of sets, because operators from $Der_{(p)(q)}\mathcal{E}$, $p,q\geq 2$, are nilpotent. The inverse map is given by the logarithmic series. Define the vector space subsheaf $Der_{k, k}\tilde{\mathcal{E}}$ of $Der_{(k)(k)}\tilde{\mathcal{E}}$ for $k\geq 0$ by
$$
\begin{array}{rl}
Der_{k, k}\tilde{\mathcal{E}} := & \{ A_{\Gamma}\in Der_{(k)(k)}\tilde{\mathcal{E}}
\mid A_{\Gamma}(\tilde{\mathcal{E}}_{r})\subset \tilde{\mathcal{E}}_{r+k},\,\,\,
\Gamma(\tilde{\mathcal{O}}_s) \subset \tilde{\mathcal{O}}_{s+k}, \,\, r,s\in \mathbb{Z}
\}.
\end{array}
$$
For an even $k\geq 2$, define a map
$$
\mu_k: \mathcal{A}ut_{(k)(2)} \tilde{\mathcal{E}} \to Der_{k,
k}\tilde{\mathcal{E}},\quad
\mu_k(a_{\gamma}) = \bigoplus_q \pr_{q+k} \circ A_{\Gamma} \circ
\pr_{q},
$$
where $a_{\gamma} = \exp (A_{\Gamma})$ and $\pr_{k}: \tilde{\mathcal{E}}
\to \tilde{\mathcal{E}}_k$ is the natural projection. The kernel of this
map is $\mathcal{A}ut_{(k+2)(2)} \tilde{\mathcal{E}}$. Moreover, the
sequence
$$
0\to \mathcal{A}ut_{(k+2)(2)} \tilde{\mathcal{E}} \longrightarrow
\mathcal{A}ut_{(k)(2)} \tilde{\mathcal{E}}
\stackrel{\mu_k}{\longrightarrow} Der_{k, k}\tilde{\mathcal{E}}\to 0,
$$
where $k\geq 2$ is even, is exact.
Denoting by $H_{(k)}(\tilde{\mathcal{E}})$ the image of the natural
mapping $H^1(M,\mathcal{A}ut_{(k)(2)}\tilde{\mathcal{E}})\to
H^1(M,\mathcal{A}ut_{(2)(2)}\tilde{\mathcal{E}})$, we get the filtration
$$
H^1(M,\mathcal{A}ut_{(2)}\tilde{\mathcal{E}}) = H_{(2)}(\tilde{\mathcal{E}})
\supset H_{(4)}(\tilde{\mathcal{E}})\supset \ldots .
$$
Take $a_{\gamma}\in H_{(2)}(\tilde{\mathcal{E}})$. We define the {\it order of}
$a_{\gamma}$ to be the maximal number $k$ such that
$a_{\gamma}\in H_{(k)}(\tilde{\mathcal{E}})$. The {\it order of a
locally free sheaf} $\mathcal{E}$ of $\mathcal{O}$-modules on a
supermanifold $(M,\mathcal{O}_M)$ is by definition the order of the
corresponding cohomology class.

\subsection{ The spectral sequence.}

A spectral sequence connecting the cohomology with values in the tangent sheaf $\mathcal{T}$
of a supermanifold $(M,\mathcal{O})$ with the cohomology with values in the tangent sheaf $\mathcal{T}_{\gr}$ of the retract $(M,\tilde{\mathcal{O}})$
was constructed in \cite{Oni_SPS}. Here we use similar ideas to construct a new spectral sequence connecting the cohomology with values in a locally free sheaf $\mathcal{E}$ on a supermanifold $(M,\mathcal{O})$ with the cohomology with values in the locally free sheaf $\tilde{\mathcal{E}}$ on $(M,\tilde{\mathcal{O}})$. Note that our spectral sequence is not a generalization of the spectral sequence obtained in  \cite{Oni_SPS} because $\mathcal{T}_{\gr}$ is not in general isomorphic to $\tilde{\mathcal{T}}$.

Let $\mathcal{E}$ be a locally free sheaf on a supermanifold
$(M,\mathcal{O})$ of dimension $n|m$. We
fix an open Stein covering $\mathfrak U = (U_i)_{i\in I}$ of $M$ and
consider the corresponding \v{C}ech cochain complex $C^*(\mathfrak
U,\mathcal{E}) = \bigoplus_{p\ge 0} C^p(\mathfrak U,\mathcal{E})$.
The $\mathbb Z_2$-grading of $\mathcal{E}$ gives rise to the $\mathbb
Z_2$-gradings in $C^*(\mathfrak U,\mathcal{E})$ and $H^*(M,\mathcal{E})$
given by
\begin{equation}\label{grading of H}
\aligned C_{\bar 0}(\mathfrak U,\mathcal{E}) &= \bigoplus_{q\ge 0}
C^{2q}(\mathfrak U,\mathcal{E}_{\bar 0})\oplus\bigoplus_{q\ge 0}
C^{2q+1}(\mathfrak U,\mathcal{E}_{\bar 1}),\\
C_{\bar 1}(\mathfrak U,\mathcal{E}) &= \bigoplus_{q\ge 0} C^{2q}(\mathfrak
U,\mathcal{E}_{\bar 1})\oplus\bigoplus_{q\ge 0}
C^{2q+1}(\mathfrak U,\mathcal{E}_{\bar 0}).\\
H_{\bar 0}(M,\mathcal{E}) &= \bigoplus_{q\ge 0}
H^{2q}(M,\mathcal{E}_{\bar 0})\oplus
\bigoplus_{q\ge 0} H^{2q+1}(M,\mathcal{E}_{\bar 1}),\\
H_{\bar 1}(M,\mathcal{E}) &= \bigoplus_{q\ge 0}
H^{2q}(M,\mathcal{E}_{\bar 1})\oplus \bigoplus_{q\ge 0}
H^{2q+1}(M,\mathcal{E}_{\bar 0}).
\endaligned
\end{equation}
The filtration (\ref{filtr_rassloenie}) for $\mathcal{E}$ gives rise
to the filtration
\begin{equation}\label{filtr ckomplex}
C^*(\mathfrak U,\mathcal{E}) = C_{(0)}\supset\ldots\supset C_{(p)}
\supset\ldots\supset C_{(m+2)} = 0
\end{equation}
of this complex by the subcomplexes
$$
C_{(p)} = C^*(\mathfrak U,\mathcal{E}_{(p)}).
$$
Denoting by $H(M,\mathcal{E})_{(p)}$ the image of the natural
mapping $H^*(M,\mathcal{E}_{(p)})\to H^*(M,\mathcal{E})$, we get the
filtration
\begin{equation}\label{filtr H}
H^*(M,\mathcal{E}) = H(M,\mathcal{E})_{(0)}\supset\ldots \supset
H(M,\mathcal{E})_{(p)}\supset \ldots.
\end{equation}
Denote by $\gr H^*(M,\mathcal{E})$ the bigraded group associated
with the filtration (\ref{filtr H}); its bigrading is given by
$$
\gr H^*(M,\mathcal{E}) = \bigoplus_{ p,q\ge 0}\gr_p
H^q(M,\mathcal{E}).
$$
By the (more general) Leray procedure, we get a spectral sequence of bigraded
groups $E_r$ converging to $E_{\infty}\simeq \gr
H^*(M,\mathcal{E})$. For convenience of the reader, we recall the main definitions here.

For any $p,r\ge 0$, define the vector spaces
$$
C^p_r = \{c\in C_{(p)}\,|\,dc\in C_{(p+r)}\}.
$$
Then, for a fixed $p$, we have
$$
C_{(p)} = C^p_0\supset\ldots\supset C^p_r\supset
C^p_{r+1}\supset\ldots.
$$
The $r$-th term of the spectral sequence is defined by
$$
E_r = \bigoplus_{p=0}^m E^p_r,\;r\ge 0, \,\, \text{where}\,\, E^p_r = C^p_r/C^{p+1}_{r-1} + dC^{p-r+1}_{r-1}.
$$
Since $d(C^p_r)\subset C^{p+r}_r$, $d$ induces a derivation $d_r$ of
$E_r$ of degree $r$ such that $d_r^2 = 0$. Then $E_{r+1}$ is
naturally isomorphic to the homology algebra $H(E_r,d_r)$.
The $\mathbb Z_2$-grading (\ref{grading of H}) in $C^*(\mathfrak
U,\mathcal{E})$ gives rise to certain $\mathbb Z_2$-gradings in $C^p_r$
and $E^p_r$, turning $E_r$ into a superspace. Clearly, the
coboundary operator $d$ on $C^*(\mathfrak U,\mathcal{E})$ is odd. It
follows that the coboundary $d_r$ is odd for any $r\ge 0$.

The superspaces $E_r$ are also endowed with a second $\mathbb
Z$-grading. Namely, for any $q\in \mathbb Z$, set
$$
\aligned
C^{p,q}_r &= C^p_r\cap C^{p+q}(\mathfrak U,\mathcal{E}),\quad
E^{p,q}_r &= C^{p,q}_r/C^{p+1,q-1}_{r-1} + dC^{p-r+1,q+r-2}_{r-1}.
\endaligned
$$
Then
\begin{equation}\label{d_r(E^(p,q)_r)subset}
E_r = \bigoplus_{p,q} E^{p,q}_r\, \, \text{and} \,\, d_r(E^{p,q}_r)\subset E^{p+r,q-r+1}_r\,\, \text{for any} \,\,r,\,p,\,q.
\end{equation}
Further, for a fixed $q$, we have $d(C^{p,q}_r) = 0$ for
all $p\ge 0$ and all $r\ge m+2$. This implies that the natural homomorphism
$E^{p,q}_r\to E^{p,q}_{r+1}$ is an isomorphism for
all $p$ and $r\ge r_0 = m+2$. Setting $E^{p,q}_{\infty} =
E^{p,q}_{r_0}$, we get the bigraded superspace
$$
E_{\infty} = \bigoplus_{p,q} E^{p,q}_{\infty}.
$$

\begin{lem}\label{E_0 and E_1=E_2} The first two terms of the spectral
sequence $(E_r)$ can be identified with the following bigraded
spaces:
$$
\aligned
E_0 = C^*(\mathfrak U,\tilde{\mathcal{E}}),\,\,
E_1 = E_2 =  H^*(M,\tilde{\mathcal{E}}).
\endaligned
$$
More precisely,
$$
\aligned
E_0^{p,q} = C^{p+q}(\mathfrak U,\tilde{\mathcal{E}}_p),\,\,
E_1^{p,q} = E_2^{p,q} =  H^{p+q}(M,\tilde{\mathcal{E}}_p).
\endaligned
$$
We have $d_{2k+1} = 0$ and, hence, $E_{2k+1} = E_{2k+2}$ for all
$k\ge 0$.
\end{lem}

\begin{proof} The proof is similar to the proof of Proposition $3$ in~\cite{Oni_SPS}.
\end{proof}

\begin{lem} There is the following identification of
bigraded algebras:
$$
E_{\infty} = \gr H^*(M,\mathcal{E}),\,\,\text{where}\,\,
E_{\infty}^{p,q} = \gr_p H^{p+q}(M,\mathcal{E}).
$$
If $M$ is compact, then
$
\dim H^k(M,\mathcal{E}) = \sum_{p+q = k}\dim  E_{\infty}^{p,q}.
$
\end{lem}

\begin{proof} The proof is a direct repetition of the proof of Proposition $4$ in~\cite{Oni_SPS}.
\end{proof}

Now we prove our main result concerning the first non-zero
co\-boundary operators among $d_2,\;d_4,\ldots$. Assume that
the isomorphisms of sheaves $\delta_i:
\mathcal{E}|U_i\to\tilde{\mathcal{E}}|U_i$ from Theorem~\ref{teor_main} are defined for each $i\in I$. By Theorem \ref{teor_main}, a locally free sheaf of
$\mathcal{O}$-modules $\mathcal{E}$ on $M$ corresponds
to the cohomology class $a_{\gamma}$ of the 1-cocycle
$((a_{\gamma})_{ij})\in Z^1(\mathfrak U, \mathcal{A}ut_{(2)(2)}
\tilde{\mathcal{E}})$, where $(a_{\gamma})_{ij} =
\delta_i\circ\delta_j^{-1}$. If the order of $(a_{\gamma})_{ij}$ is
equal to $k$, then we may choose $\delta_i,\; i\in I$, in such a way
that $((a_{\gamma})_{ij})\in Z^1(\mathfrak U,\mathcal{A}ut_{(k)(2)}
\tilde{\mathcal{E}})$. We can write $a_{\gamma} = \exp A_{\Gamma}$, where
$A_{\Gamma}\in C^1(\mathfrak U,Der_{(k)(2)} \tilde{\mathcal{E}})$.

We will identify the superspaces $(E_0,d_0)$ and $(C^*(\mathfrak
U, \tilde{\mathcal{E}}),d)$ via the isomorphism of Lemma
\ref{E_0 and E_1=E_2}. Clearly, $\delta_i: \mathcal{E}_{(p)}|U_i\to
\tilde{\mathcal{E}}_{(p)}|U_i = \sum_{r\ge p}\tilde{\mathcal{E}}_r|U_i$ is
an isomorphism of sheaves for all $i\in I,\; p\ge 0$. These local
sheaf isomorphisms permit us to define an isomorphism of graded
cochain groups
$$
\psi: C^*(\mathfrak U,\mathcal{E})\to C^*(\mathfrak U,\tilde{\mathcal{E}})
$$
such that
$$
\psi: C^*(\mathfrak U,\mathcal{E}_{(p)})\to C^*(\mathfrak
U,(\tilde{\mathcal{E}})_{(p)}),\;p\ge 0.
$$
We put
$$
\psi(c)_{i_0\ldots i_q} = \delta_{i_0}( c_{i_0\ldots i_q})
$$
for any $(i_0,\ldots,i_q)$ such that $U_{i_0}\cap\ldots\cap
U_{i_q}\ne \emptyset$. Note that $\psi$ is not an isomorphism of
complexes. Nevertheless, we can explicitly express the coboundary
$d$ of the complex $C^*(\mathfrak U,\mathcal{E})$ by means of $d_0$ and
$a_{\gamma}$.

The following theorem permits to calculate the spectral
sequence $(E_r)$ whenever $d_0$ and the cochain $a_{\gamma}$ are
known. It also describes certain coboundary operators
$d_r,\; r\ge 1$.

\begin{thm} For any $c\in C^*(\mathfrak
U,\tilde{\mathcal{E}}_q) = E_0^q$, we have
$$
(\psi(d\psi^{-1}(c)))_{i_0\ldots i_{q+1}} = (d_0c)_{i_0\ldots
i_{q+1}} + ((a_{\gamma})_{i_0i_1} -\id )(c_{i_1\ldots i_{q+1}}).
$$

Suppose that the locally free sheaf of
$\mathcal{O}$-modules $\mathcal{E}$ on $M$ has order $k$
and denote by $a_{\gamma}$ the cohomology class corresponding to
$\mathcal{E}$ by Theorem \ref{teor_main}. Then $d_r = 0$ for $r =
1,\ldots,k-1$, and $d_{k} = \mu_k(a_{\gamma})$.
\end{thm}
\begin{proof} The proof is similar to the proof of Theorem $7$ in \cite{Oni_Vi}.\end{proof}

\subsection*{Acknowledgment}
The author is very grateful to A.\,L.\,Onishchik for useful discussions.

\end{document}